\documentclass[12pt,reqno]{amsart}
\usepackage{latexsym,amssymb,amsmath,amsfonts,amsthm,url}
\usepackage[utf8]{inputenc}
\usepackage[english]{babel}
\usepackage{enumitem}

\usepackage{diagbox}
\usepackage{bbm}
\usepackage[dvipsnames]{xcolor}

\newtheorem{theorem}{Theorem}

\newtheorem{lemma}{Lemma}
\newtheorem{remark}{Remark}

\theoremstyle{definition}

\newcommand{\abs}[1]{\left\vert#1\right\vert}

\renewcommand\hat{\widehat}

\def\beq{ \begin{equation} }
\def\eeq{ \end{equation} }

\def\square{\vcenter{\vbox{\hrule height .4pt
			\hbox{\vrule width .4pt height 5pt \kern 5pt
				\vrule width .4pt} \hrule height .4pt}}}

\def\sqz{\kern -0.2em}

\def\E{\mathbb{E}}

\def\hat{\widehat}

\def\1{\mathbbm{1}}
\def\P{\mathbb{P}}

\def\Px{\P ^{(x)}}

\def\ER{Erd\H{o}s-R\'enyi}
\usepackage{graphicx}
\graphicspath{%
	{converted_graphics/}
	{/}
}
\title{A note on the cooperative two-type SIR processes on Galton-Watson trees}

\author[Ma]{Ruibo Ma}

\author[Liu]{Tai Heng Liu}

\author[Baghdadi]{Baghdadi Othmane}

\author[Yao]{Dong Yao}

\address{School of Mathematics and Statistics, Beijing Jiaotong University.}
\email{9908@bjtu.edu.cn}

\address{Department of Mathematics, University of Hawai`i at M\={a}noa.}
\email{taihengliu2@gmail.com}

\address{LaMAO, Faculty of Sciences, Mohammed First University, Oujda, Morocco}
\email{othmane.baghdadi.d23@ump.ac.ma}

\address{(Corresponding author) School of Mathematics and Statistics/RIMS, Jiangsu Normal University.}
	\email{dongyao@jsnu.edu.cn}

\begin{document} 
\begin{abstract}
In the standard SIR model on a graph, infected vertices infect their neighbors at rate $\alpha$ and recover at rate $\mu$. We  
 consider a two-type SIR process where each individual in the graph can be infected with two types of diseases, $A$ and $B$. Moreover, the two diseases interact in a cooperative way so that an individual that has been infected with one type of disease can acquire the other at a higher rate. We prove that if the underlying graph is a Galton-Watson tree and initially the root is infected with both $A$ and $B$, while all others are susceptible, then the two-type SIR model has the same critical value for the survival probability as the classic single-type model.  
\end{abstract}

    \maketitle

\textit{Classification: }60J27, 92D30

	\section{Introduction}


SIR models and their variants have been widely used to predict and control the disease spreading process.  In the standard single-type SIR process based on a graph, each node can be in one of three states: S (representing `susceptible'), I (representing `infected') and R (representing `recovered'). Vertices in state I try to infect their neighbors at rate $\alpha$, independently across each edge, and recover (turn into state R) at rate $\mu$. Alternatively, in terms of species evolution, one might imagine that each individual gives birth after a random amount of time distributed as Exp($\alpha$)  and dies after time distributed as Exp($\mu$), where all the exponential random variables are independent. A closely related model is the SIS model, also called the contact process, where a vertex turns into state S instead of R when it recovers.

Researchers have also been fascinated by the behavior of epidemic models that incorporate interactions (competition or cooperation) among multiple species.  Two-type SIS models with competitions
for space on the $d$-dimensional integer lattice $\mathbb{Z}^d$
were investigated in \cite{MounTford2019,Neuhauser1992,Stover2020}.
For the case of equal death rates,  Neuhauser \cite{Neuhauser1992} proved that the species with the smaller birth rate dies out locally, while 
Mountford,  Pantoja, and Valesin \cite{MounTford2019}  proved that the winner takes over a ball whose radius grows linearly over time.
Neuhauser conjectured that for general death rates, the species with the higher birth/death ratio wins the competition, which has been verified by Stover \cite{Stover2020} for certain cases but still remains largely open.
For cooperative interactions, Durrett and Yao \cite{DurrettYao2020} considered a symbiotic contact process on $\mathbb{Z}^d$
where the presence of one species can reduce the death rate of the other type at the same site. It was proved that strong symbiosis can lower the critical value, but the general case (especially weak  symbiosis) remains unclear.

In another direction, Lanchier and Neuhauser
\cite{lanchier2010stochastic,lanchier2006stochastic} studied stochastic models with hosts and symbionts. In these models, each host can be infected with a symbiont, while a few species of hosts compete against each other.  Durrett and Lanchier \cite{durrett2008coexistence} studied another case where only one of two species of hosts can be infected with one type of symbionts. In these papers, conditions for survival and coexistence have been studied, and the main results describe the long-term behavior of the models under certain conditions.
A few years later, Lanchier and Zhang \cite{lanchier2016some} studied the ``stacked contact process'', and Ma \cite{ma2022complete} studied the ``two-level contact process''. In these two models, there are uninfected hosts and infected hosts. The results focused on the phase transition and the limiting distribution of the models.

	Motivated by a series of works by Ghanbarnejad and coauthors
    \cite{cai2015avalanche, chen2013outbreaks,grassberger2016phase,zarei2019exact},
    we consider a \emph{two-type cooperative SIR process} on a network.  Suppose that two diseases, $A$ and $B$, spread simultaneously on the network, each type acting as an SIR process. Vertices in the graph can be infected by both types. Thus, there are nine possible states for each vertex,
    \begin{equation}\label{eq:statespace}
            S\text{, }A \text{, } B \text{, } a \text{, } b \text{, } Ab \text{, } aB \text{, } AB\text{ and }ab.
    \end{equation}
    Here, capital $A$ and $B$ indicate active infection, and  lower case $a$ and $b$ indicate recovery from $A$ and $B$. We assume that the two diseases interact in a cooperative way. 
    When a vertex $x$  is susceptible to both diseases, then it can be infected with $A$ at rate $\alpha_1$. On the other hand, 
    if  $x$ has been infected with $B$ before (regardless of whether it has recovered from $B$ or not), then $x$ acquires  $A$ at a higher rate  $\beta_1 \in [\alpha_1, \infty]$.
    Note that an infinite rate means that an infection transmits immediately.  The constant $C _ 1:= \beta _1 / \alpha _1 \in [1,\infty]$ is called the \emph{cooperativity coefficient} for  $A$. Likewise, a vertex $x$ can be infected by $B$ at a higher rate $\beta _2 \geq \alpha _2$ if it has been infected with $A$ before. The cooperativity coefficient of $B$ is $C _ 2 = \beta _2 / \alpha _2$. Our definition of the two-type SIR process simulates the situation in which the immune system weakens because of a prior infection.
    

Chen et al.\@ \cite{chen2013outbreaks} analyzed a deterministic version of this model with mean-field methods, involving a homogeneously mixed population of infinitely many agents. 
In total, there are nine ordinary differential equations for the evolution of all states. Let $[z](t)$ be the fraction of agents in state $z$ at time $t$.
Under the assumption of equal infection rates for both diseases, and letting the recovery rate be 1, \cite{chen2013outbreaks} reduced the system of nine equations to three:
\begin{equation}    \label{eq:ode}
   \begin{cases}
      s'(t)=-2\alpha s(t)x(t),  \\  
  q'(t)=(\alpha s(t)-C \alpha q(t))x(t), \\
  x'(t)=(\alpha s(t)+C \alpha q(t))x(t)-x(t). 
    \end{cases}
\end{equation}
In these equations, $s(t)=[S](t)$ is the fraction of susceptible agents, 
$$x(t)=[A](t)+[AB](t)+[Ab](t)=[B](t)+[AB](t)+[aB](t)$$
is the fraction of agents that are actively infected  by disease $A$  (or $B$) at time $t$, and $$q(t)=[A](t)+[a](t)=[B](t)+[b](t)$$ accounts for the agents which have  an infection history of only one disease. The initial condition is set to $[A](0)=[B](0)=\epsilon/2$, $[S] (0) = 1 - \epsilon$ and the fractions of all other states are $0$. Hence, $s(0)=1-\epsilon$ and $q(0)=x(0)=\epsilon/2$.   

Denote the final epidemic size of the system of equations \eqref{eq:ode} by $$\mathcal{R}(\alpha, C, \epsilon) = 1-\lim_{t\to\infty} s(t).$$
Zarei et al.\@  \cite{zarei2019exact} found
$\mathcal R(\alpha, C, \epsilon)=1-s(0)\exp(-2\alpha T_0)$, where $T_0$ is given by
$$
\inf\left\{t>0: t+s(0)\exp(-2\alpha t)+q(0)\exp(-C\beta t)-\frac{ s(0)}{C-2}(\exp(-C\alpha t)-\exp(-2\alpha t) )=1  \right\}
$$
if $C\neq 2$.
Note that this expression already implies the criticality at $C=2$.
Let
$$
\mathcal{R}_*(\alpha, C) = \lim_{\epsilon\to 0} \mathcal R(\alpha, C, \epsilon)
$$
when the limit exists (otherwise, one may replace the limit with upper limit). Based on numerical experiments and some non-rigorous arguments,  \cite{chen2013outbreaks, zarei2019exact} claimed the following:
\begin{enumerate}[label = (\roman*)]
	\item If $C \leq 2$, then $\mathcal R(\alpha, C, \epsilon)$ is continuous in $\alpha$. Moreover, $\mathcal R_*(\alpha, C)=0$ for $\alpha \leq 1$, and behaves like $\alpha-1$ ($C<2$) or $\sqrt{\alpha-1}$ ($C=2$) for $\alpha$ close to 1.
	\item If $C>2$, then $\mathcal R(\alpha, C, \epsilon)$  is discontinuous in $\alpha$ at some $\alpha_0=\alpha_0(C,\epsilon)$, and
\begin{equation}\label{eq:discont}
    \lim_{\epsilon \to 0} \mathcal R(\alpha_0-, C, \epsilon)=0 \quad \textnormal{ and } \quad 
\lim_{\epsilon \to 0} \mathcal R(\alpha_0+, C, \epsilon)>0. 
\end{equation}
\item For $C>2$,  the critical infection rate $\alpha_0(C, \epsilon)=1-\sqrt{(C-2)\epsilon}+O(\epsilon)$ as $\epsilon\to 0$.
\item The quantity $\mathcal R_*(\alpha, C)\equiv 0$ for $\alpha \leq 1/2$, and  $$\lim_{\epsilon \to 0}\lim_{C\to\infty}R(\alpha, C,\epsilon)  \sim \alpha-\frac{1}{2}\quad  \textnormal{ as } \alpha \to 1/2 \textnormal{.}$$
\end{enumerate}

It is generally expected that
mean-field ODEs may give good approximations for the stochastic particle systems  on the complete graph.
 In real-world networks, nontrivial spatial structures are often present.  Grassberger et al.\@ \cite{grassberger2016phase} conducted simulations for the two-type SIR model on  Erd\H{o}s-R\'enyi graphs and the integer lattice $\mathbb{Z}^d$. They argued that 
 Erd\H{o}s-R\'enyi graphs have a discontinuous 
 phase transition for the fraction of eventually infected vertices when the cooperativity is sufficiently strong,  as predicted by the mean-field ODEs. For the case of $\mathbb{Z}^d$ with $d\geq 2$, simulations in  \cite{grassberger2016phase}  showed that the two-type model shares similar critical and near-critical features  with the  single-type SIR process if and only if $d\leq 3$. 
 \cite{grassberger2016phase}  also considered   the near-critical behavior of the two-type SIR model on trees. In particular,  \cite{grassberger2016phase}  suggested that the critical value (for infinitely many vertices to be infected with both diseases) remains the same as the single-type model for $C=\infty$. See also Remark \ref{rmk:c=infty} below.  
 
 
 In this paper, we consider the two-type SIR process on Galton-Watson trees with cooperative interactions  defined before. The rates are not necessarily the same for $A$ and $B$, and there are six parameters in total, shown in Table \ref{rates table}.

    \begin{table}[h]
       \label{rates table}
\begin{center}
\begin{tabular}{c|c|c|c}
\hline
& original infection rate   & increased infection rate & recovery rate\\
\hline 
disease $A$ &  $\alpha_1$  & $\beta_1\in [\alpha_1,\infty]$ & $\mu_1$\\
\hline 
disease $B$ &  $\alpha_2$  & $\beta_2\in [\alpha_2, \infty]$  & $\mu_2$\\
\hline 
\end{tabular}
\end{center}
\caption{Parameters for the two-type SIR process.}
    \end{table}

Let $GW (p)$ be the Galton-Watson tree with the offspring distribution $p$. Let  $m \in (1,\infty)$ be the finite mean of $p$. (We let $m>1$ so that $GW(p)$ itself is supercritical.) Consider the single-type SIR process on $GW(p)$. We say that it \textit{survives} if for all $t\geq 0$, there is at least one infected vertex at time $t$. The following result is classic. 
   \begin{theorem}\label{thm:single}
   Let the infection rate be $\alpha$ and the recovery rate be $\mu$ for a single-type SIR process on $GW (p)$. At time 0, only the root is infected. The probability of survival is greater than 0 if and only if   $   \alpha/\mu > 1/{(m-1)}   $, where $m$ is the mean of $p$. Thus, if $\mu$ is fixed, then the critical value for $\alpha$ is  $\mu/(m-1)$. 
   \end{theorem}
Given Theorem \ref{thm:single}, it is natural to probe whether the critical value gets smaller in the cooperative SIR process. We say that the two-type SIR process survives if for all $t\geq 0$, there is at least one infected vertex (with $A$ or $B$) at time $t$. According to the following Theorem \ref{general}, the answer is no.

\begin{theorem}
    \label{general}
Consider the two-type SIR process on $GW (p)$ with parameters $\alpha_i$, $\beta_i$, $\mu_i$, $i=1,2$. At time $0$, the root is infected with both disease $A$ and disease $B$. All other vertices are susceptible to both $A$ and $B$. Let $m$ be the mean of $p$. If
    \begin{equation}\label{eq:assump}
      \max\left\{\frac{\alpha_1}{\mu_1},\frac{\alpha_2}{\mu_2}\right\}\leq \frac{1}{m-1}
      \textnormal {,}
    \end{equation}
then the probability of survival is $0$.
\end{theorem} 
\begin{remark}\label{rmk:c=infty}
In the special case of equal infection/recovery rates for both diseases (denoted by $\alpha,\beta=C\alpha$ and $\mu$)  with $C=\infty$,  the conclusion of    Theorem \ref{general}  is consistent with  the prediction made in \cite{grassberger2016phase} regarding infinite tree graphs, where the authors claimed that (for $C=\infty$), 
$$
\mathbb{P}(\mbox{infinitely many vertices are infected with both diseases})>0\, \mbox{ if and only if }\, \frac{\alpha}{\mu}>\frac{1}{m-1}. 
$$ See also the discussions above Table       \ref{rates table}.
\end{remark}
Combining Theorems \ref{thm:single} and \ref{general}, the two-type SIR process survives with a positive probability if and only if
  $$
      \max\left\{\frac{\alpha_1}{\mu_1},\frac{\alpha_2}{\mu_2}\right\}> \frac{1}{m-1}.
  $$
 Since the Galton-Watson tree and its variants arise as local limits of some random graph models such as the configuration model and the {\ER} graph, our result may be useful for studying the two-type SIR processes on these random graphs. We leave further investigations as  future work.   

In the remainder of this paper, we prove
Theorems \ref{thm:single} and  \ref{general}  in Section \ref{sec:single} and Section \ref{sec:two}, respectively, by coupling the Galton-Watson tree with the epidemic process. 
Section \ref{sec:conclusion} summarizes the main findings of this paper and states many open problems as well as some simulation results.

    \section{Single-type SIR on Galton-Watson trees}       \label{sec:single}

    Though Theorem \ref{thm:single} is more or less well-known, here we give a proof for the sake of completeness and also to illustrate the basic idea that will also be used in the proof of Theorem \ref{general}. We couple the Galton-Watson tree with an epidemic process by 
 revealing the number of children of a vertex after it becomes infected. At time 0 only the root $o$ is infected and 
    the only available information about the tree  is the degree of $o$. Since $o$ infects each of its children at rate $\alpha$ and recovers at rate $\mu$, the probability that any given child is infected before $o$ recovers is equal to $\alpha/(\alpha+\mu)$ by standard properties of the exponential distribution. 
    For any vertex $x$,
    we let $N_x$ be the number of children of $x$ that are eventually infected by $x$.   Since the degree $D_o$ of $o$ is distributed according to $p$, we see that the mean of $N_o$ is 
    \begin{equation}\label{eq:edo}
        \E\left( \frac{\alpha}{\alpha + \mu} D_o\right)=\frac{m\alpha}{\alpha+ \mu}.
    \end{equation}
    Denote the distribution of $N_o$ by $q$. 
    Since the number of children of any given vertex has the same distribution $p$, we see that the number of infected children of any infected vertex must have the distribution $q$, and must be independent of its ancestors.
    
    Let $Y_n$ be the number of infected vertices in the $n$-th generation of $GW (p)$. Then
$Y_n$ forms a branching process with the initial value $Y_0=1$ and branching distribution $q$, whose 
mean is equal to $m\alpha/(\alpha+ \mu)$.
    Standard results on the branching process imply that 
    \begin{equation}\label{eq:thm}
    \P(Y_n>0, \forall n)>0 \Leftrightarrow 
    \frac{\alpha}{\alpha+ \mu}m>1 \Leftrightarrow \alpha>\frac{ \mu}{m-1}
    \textnormal  {,}
    \end{equation}
which proves Theorem \ref{thm:single}.

    \section{Two-type SIR on Galton-Watson trees}
        \label{sec:two}

    We present the proof of Theorem \ref{general} in this section.  Lemma \ref{t} below is the key step.

\begin{lemma}\label{t}
    Assume all conditions in  Theorem \ref{general}. For all vertices $x$ in $GW (p)$,
    let $\tau_A(x)$ and $\tau_B(x)$ be the time
    that $x$ gets infected with $A$ and $B$, respectively. (Note that if the vertex $x$ is never infected by a disease, we let the corresponding time be $\infty$.)
    Let $y$ be a child of $x$.
    The conditional probability given
    $ \sigma (\tau_A(x), \tau_B(x))$ that $y$ is infected with both $A$ and $B$ is almost surely less than or equal to $1/m$.
\end{lemma}

    \begin{proof} 
    If $\max\{\tau_A(x), \tau_B(x)\} = \infty$, then the vertex $x$ or its child $y$ cannot receive both infections. Hence, we only consider the   case where $\max\{\tau_A(x), \tau_B(x)\}<   \infty$.  Without loss of generality, we assume that $\tau_A(x)\leq \tau_B(x)$. 
    Define four events:
    \begin{equation*}
        \begin{aligned}
                  \Omega_1 &= \{x \text{ does not infect }y 
            \text{ with disease } A \text{ or recover from } A \text{ before time } \tau_B(x)\} \text {,} \\
            \Omega_2 &= \{ x \text{ infects }y 
            \text{ with disease } A \text{ before time } \tau_B(x) \} \text {,} \\
            \Omega_3 &= \{\text{both diseases are transmitted to } y \text{ after time }\tau_B(x)\}  \text {,}  \\
            \Omega_4 &= \{ x \text{ infects }y 
            \text{ with disease } B \text{ after time } \tau_B(x) \}          \text{.}   
        \end{aligned}
    \end{equation*}
There are two possible ways that  $y$ gets both infections:
\begin{enumerate}   [label = (\roman*)]
    \item
        Case 1: $\Omega_1$ and
        $\Omega_3$ both occur;
    \item
Case 2:        $\Omega_2$ and $\Omega_4$ both occur. 
\end{enumerate}

Let $\P ^{(x)}$ denote the conditional probability on $ \sigma (\tau_A(x), \tau_B(x))$, i.e., for any event $\Omega$, $\P ^{(x)} (\Omega) = \P (\Omega | \sigma (\tau_A(x), \tau_B(x)))$. Let $P_1 = \Px (\Omega_1\cap \Omega_3)$, and $P_2 = \Px (\Omega_2\cap \Omega_4)$. Let $P_0$ be the conditional probability considered in this lemma. Then, $P_0 = P_1 + P_2$ a.s. For Case 1,  
\begin{equation}\label{eq:p1}
    P_1 = \Px(\Omega_1)\P(\Omega_3|\Omega_1) \text{ a.s.}
\end{equation}
Since the minimum of the infection time and recovery time of $A$  has distribution Exp($\alpha_1+\mu_1$), we know that almost surely
\begin{equation}\label{eq:pomega1}
    \Px (\Omega_1)=\exp\left(-(\mu_1+\alpha_1)(\tau_B(x)-\tau_A(x))\right) \text{.}
\end{equation}

Given $\Omega_1$, at time $\tau_B(x)$, $x$ has infections $A$ and $B$, while $y$ is susceptible to both diseases. Four events may happen next: $y$ gets infection $A$, $y$ gets infection $B$, $x$ recovers from $A$ and $x$ recovers from $B$. For $\Omega_3$ to occur, the first of these four events to occur must be one of the first two. Given the rates of the four events, $\alpha_1$, $\alpha_2$, $\mu_1$ and $\mu_2$,          
\begin{equation}\label{eq:omega3}
    \P(\Omega_3|\Omega_1)\leq     \frac{\alpha_1+\alpha_2}{\alpha_1+\mu_1+\alpha_2+\mu_2}
    \leq\frac{\alpha_1+\alpha_2}{\alpha_1+(m-1)\alpha_1+\alpha_2+(m-1)\alpha_2}=\frac{1}{m}.     
\end{equation}

By \eqref{eq:p1}, \eqref{eq:pomega1} and \eqref{eq:omega3}, we get
\begin{equation}\label{eq:p1bd}
    P_1 \leq \frac{1}{m}  \exp\left(-(\mu_1+\alpha_1)(\tau_B(x)-\tau_A(x))\right)   \text{ a.s.}
\end{equation}

For Case 2,  we see that
\begin{equation}\label{eq:p2}
P_2 = \Px (\Omega_2\cap \Omega_4) \leq\Px (\Omega_2)\text{ a.s.,}
\end{equation}
which is further bounded by
\begin{equation}\label{eq:pomega2}
    \begin{aligned}
        \Px (\Omega_2) &= \frac{\alpha_1}{\alpha_1+\mu_1} \left(1-\exp\left(-(\mu_1+\alpha_1)(\tau_B(x)-\tau_A(x))\right)\right)   \\
        &\leq   \frac{1}{m}\left(1-\exp\left(-(\mu_1+\alpha_1)(\tau_B(x)-\tau_A(x))\right)\right)     \text{ a.s.}
    \end{aligned}
\end{equation}

Combining  \eqref{eq:p1bd}, \eqref{eq:p2} and \eqref{eq:pomega2}, we see that 
$$
P_0=P_1+P_2\leq  \frac{1}{m}  \text{ a.s.}
$$
This concludes the proof of Lemma \ref{t}. 
\end{proof}

\begin{proof}[Proof of Theorem \ref{general}]

Let $\mathcal{X}_n$ be the set of vertices in the $n$-th generation of $GW (p)$ that eventually receive both infections, and let $X_n=\abs{\mathcal{X}_n}$, for all $n \geq 0$. Note that the root is the single vertex in the zeroth generation. Let $\mathcal{F}_n$ be the $\sigma$-algebra generated by the number of children of all $x$ up to the $ (n-1)$-th generation and all $\tau_A(x)$, $\tau_B(x)$, for $x$ up to the $n$-th generation.  By Lemma \ref{t},  
\begin{equation*}
    \E(X_{n+1}|\mathcal{F}_n)\leq X_n\textnormal{ a.s., for }n\geq 0.
\end{equation*}
Thus $\{X_n,n\geq 0\}$ is a non-negative integer-valued supermartingale. Hence $X_n$ must converge almost surely to some limit $X$. Moreover, there exists a constant $c>0$ such that for all $k,n\geq 1$,
\begin{equation}\label{xn=0}
\P(X_{n+1}=0|X_n=k)\geq c^k,
\end{equation}
by considering the event where all vertices in $\mathcal{X}_n$ recover before infecting any child. Let $\hat{\Omega}$ be the event
    $$
        \left\{   \lim_{n\to\infty} X_n=0   \right\}=\{X_n=0\textnormal   { for some }n\}.
    $$
By Levy's 0-1 Law \cite[Theorem 4.6.9] {durrett2019probability} and equation \eqref{xn=0},
    \begin{equation*}
        \P ( \mathbf{1}_{\hat{\Omega}}  \mid      \mathcal{F}_n) \to 1=1_{\hat{\Omega}} \text{ a.s.}
    \end{equation*}
Thus, $\P(\hat{\Omega})=1$ and $X=0$ a.s. Now define
\begin{equation*}
    N=\inf\{n: X_n=0\}.
\end{equation*}
All vertices $z$ in the $N$-th generation of $GW (p)$ are either infected by one type of disease, or never infected.    Since the subtree of $GW (p)$ re-rooted at $z$ has the same distribution as $GW (p)$, we can apply Theorem \ref{thm:single} to deduce that the number of infected vertices in that subtree is almost surely finite. Thus, we conclude that the two-type SIR process 
survives with  probability 0.
\end{proof}

\section{Conclusion and future directions}\label{sec:conclusion}
In this paper, we studied the behavior of the two-type cooperative SIR process on a Galton-Watson tree without any symmetry assumption on the infection and recovery rates. We proved that the critical value remains the same as
in the single-type SIR model, regardless of the magnitude of cooperativity coefficients.

We now list a few open questions for future research.

\begin{itemize}
    \item \textbf{Rigorous analysis for the ODE system \eqref{eq:ode}.} It would be interesting to inspect the claims made by physicists \cite{chen2013outbreaks} regarding  properties of final epidemic size $\mathcal R(\alpha, C, \epsilon)$, particularly the discontinuous transition \eqref{eq:discont}. One can also study whether the mean-field equations approximate the true dynamics on the complete graph. 
    \item \textbf{Survival of the weaker species in the asymmetric case.}  If one species is (sub)critical while the other is supercritical, can the weaker species survive with positive probability  when the cooperation is strong?   See  simulations (Figures \ref{sim growth} and \ref{sim_surv}) and discussions   below.
    \item \textbf{Effects of the graph structure on the two-type SIR process.} As mentioned in the introduction, the simulations in \cite{grassberger2016phase} found that the structure of the underlying graph (particularly the existence of loops in the graph) has a major impact on the two-type SIR process thereon.  We  do expect a qualitatively different situation for the survival probability if the Galton-Watson tree is replaced with the integer lattice.
    \item \textbf{Different cooperativity mechanism.}
    One can also consider the cooperativity mechanism as in \cite{DurrettYao2020}, where the recovery rates decrease from $\mu_i$ to $\hat{\mu}_i$ if a node has two infections. By comparison with two independent single-type SIR processes,  it can be shown that  the critical value for survival matches the single-type model if $\hat{\mu}_i$ is sufficiently close to $\mu_i$. However, the case where $\hat{\mu}_i$ is small remains unknown.
    \item \textbf{Cooperative SIS models.} If we consider the SIS dynamics on a Galton-Watson tree, by an oriented percolation argument, the critical value becomes smaller if the cooperation is sufficiently strong. As indicated by \cite{DurrettYao2020}, the general case may be rather challenging. 
\end{itemize} 

We use a computer simulation to investigate the second open problem mentioned in the list. Suppose that the Galton-Watson tree is a binary tree, and consider the parameters shown in Table \ref{tab:simpa}. Now, the original rate of $A$ is supercritical, while the original rate of $B$ is subcritical. We simulate the two-type SIR process for a few choices of $\beta _2$ to predict whether a sufficiently large value of $\beta_2$ can trigger supercritical behavior of the disease $B$.

In our simulation, five out of six parameters are fixed, while $\beta_2$ ranges from $1$ to $1.8$. For each value of $\beta_2$, the simulation is run $200$ times, up to the 18th generation of the binary tree. The results are presented by calculating the average number of nodes infected with $B$ in each generation and by computing the proportion of simulations in which the disease $B$ survives up to each generation.

    \begin{table}[!htbp]
\begin{center}
\begin{tabular}{c|c|c|c}
\hline
& original infection rate   & increased infection rate & recovery rate\\
\hline 
disease $A$ &  5  & 8 & 1    \\
\hline 
disease $B$ &  0.75  & 1, 1.2, 1.4, 1.6, 1.8 & 1    \\
\hline 
\end{tabular}
\end{center}
\caption{Parameters used in the simulation.}
\label{tab:simpa}
    \end{table}

The results are shown in Figures \ref{sim growth} and \ref{sim_surv}. They suggest that a phase transition in $\beta_2$ exists. For $\beta_2 \leq 1.2$, it is likely that the disease $B$ stays subcritical, while for $\beta_2 \geq 1.6$, the disease $B$ appears to become supercritical, infecting more and more nodes in later generations.  A rigorous analysis of this phenomenon  is left as  future work.

\begin{figure}[!htbp]
    \centering
    \includegraphics[width = 0.9 \textwidth]{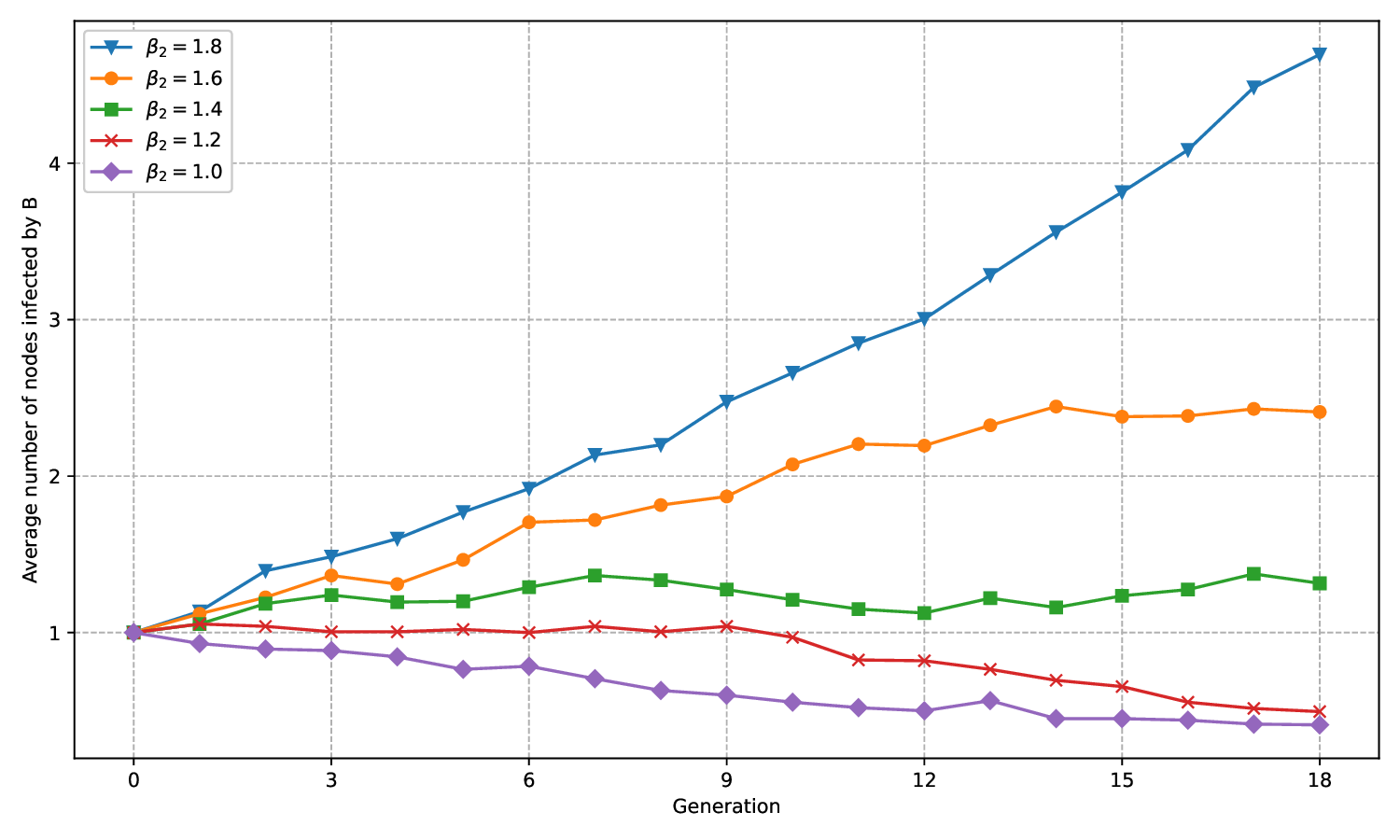}
    \caption{The average number of $B$ infections in each generation. Five parameters are fixed:  $\alpha _1 = 5$, $\alpha_2 = 0.75$, $\beta_1 = 8$, and $\mu _1 = \mu _2 = 1$. The value of $\beta_2$ varies from 1.0 to 1.8.}
    \label{sim growth}
\end{figure}

\begin{figure}[!htbp]
    \centering
    \includegraphics[width = 0.9\textwidth]{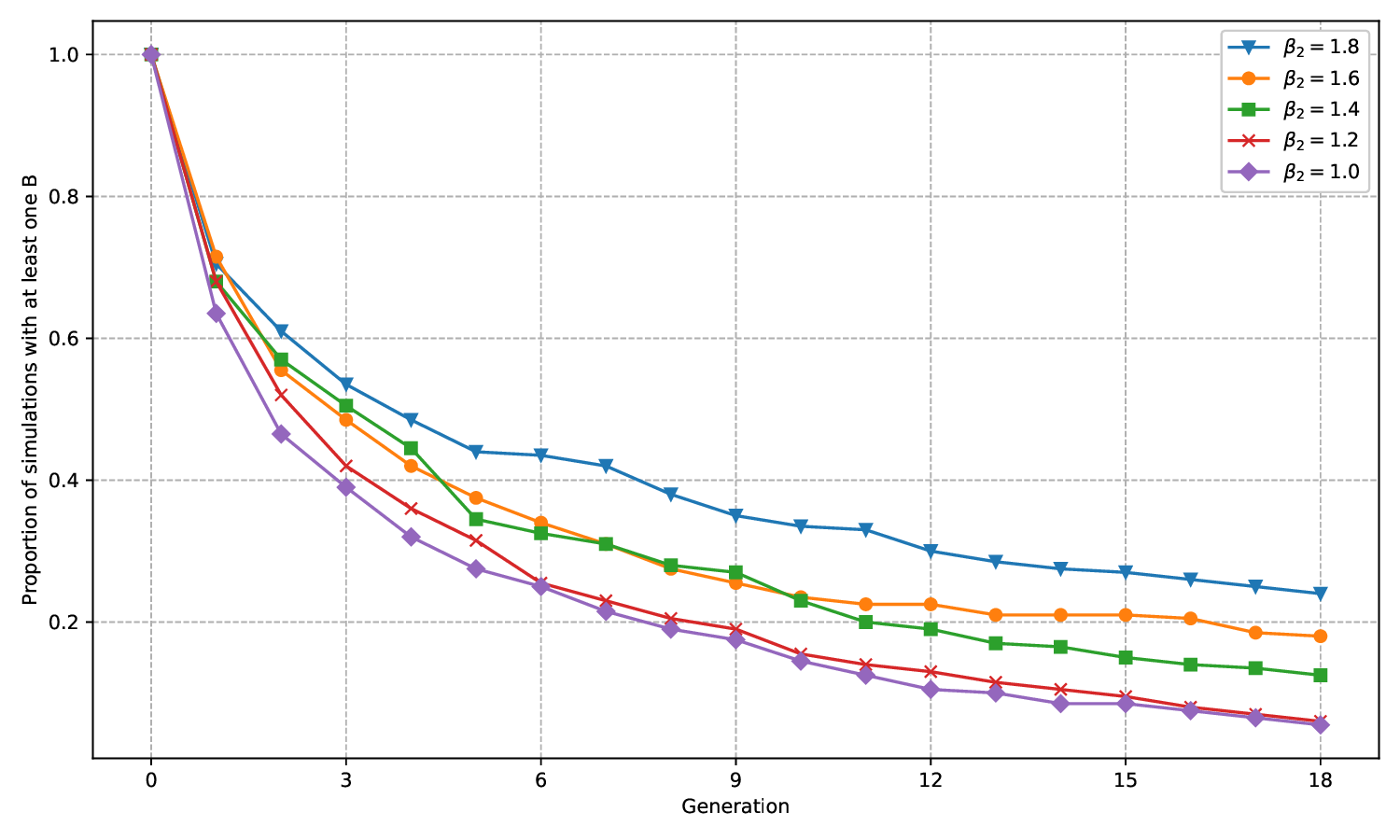}
    \caption{The proportion of simulations with some $B$ infections in each generation. Five parameters are fixed:  $\alpha _1 = 5$, $\alpha_2 = 0.75$, $\beta_1 = 8$, and $\mu _1 = \mu _2 = 1$. The value of $\beta_2$ varies from 1.0 to 1.8.}
    \label{sim_surv}
\end{figure}


\section*{Acknowledgements}

Ruibo Ma is supported by the Talent Fund of Beijing Jiaotong University 2023XKRC025. 
Dong Yao is supported by NSFC No.\@ 12201256 and  Basic Research Program of  Jiangsu grant No.\@ BK20220677. 
Dong Yao would like to thank Prof.\@ Rick Durrett for suggesting this problem to him. The authors also thank Jiawei Li and Lijun Tu for helpful discussions, and the anonymous referee for constructive comments.

	\bibliography{ref}
	\bibliographystyle{plain}

\end{document}